\documentclass[11pt]{article}

\usepackage[a4paper,margin=25mm]{geometry}
\usepackage[T1]{fontenc}

\usepackage{graphicx} 
\usepackage{tikz}
\usetikzlibrary{shapes.multipart, positioning}

\usepackage{color}

\title{Strong invariants and Tverberg numbers in convexity spaces}

\author{
Minho Cho\thanks{School of Computational Sciences, Korea Institute for Advanced Study, Seoul, South Korea. \texttt{minhocho.math@gmail.com}. Supported by a KIAS Individual Grant (CG102101) at Korea Institute for Advanced Study.}
\and
Andreas F. Holmsen\thanks{Department of Mathematical Sciences, KAIST, Daejeon, South Korea; Discrete Mathematics Group, Institute for Basic Sciences, Daejeon, South Korea. \texttt{andreash@kaist.edu}. Supported by the Institute for Basic Science (IBS-R029-C1).}
\and
Attila Jung\thanks{ELTE E\"otv\"os Lor\'and University, Budapest, Hungary; Alfr\'ed R\'enyi Institute of Mathematics, Budapest, Hungary.
\texttt{jungattila@gmail.com}. Supported by the ERC Advanced Grant ``ERMiD'', and by the EXCELLENCE-24 project no.~151504 of the NRDI Fund.}
\and 
Hong Liu\thanks{Extremal Combinatorics and Probability Group, Institute for Basic Science, Daejeon, South Korea. \texttt{hongliu@ibs.re.kr}. Supported by the Institute for Basic Science under grant IBS-R029-C4.}
}

\date{}

\usepackage{amsfonts,amsmath,amssymb,amsthm}
\usepackage{comment}
\usepackage{mathtools}
\usepackage{todonotes}

\newtheorem{thm}{Theorem}[section]
\newtheorem{lem}[thm]{Lemma}

\newtheorem{prop}[thm]{Proposition}

\newtheorem{cor}[thm]{Corollary}

\newtheorem{claim}[thm]{Claim}

\theoremstyle{definition}
\newtheorem{example}[thm]{Example}

\newtheorem{definition}[thm]{Definition}
\newtheorem{observation}[thm]{Observation}
\newtheorem{rem}[thm]{Remark}
\newtheorem{question}[thm]{Question}

\newcommand{\F}{\mathcal{F}}
\newcommand{\C}{\mathcal{C}}
\newcommand{\B}{\mathcal{B}}
\newcommand{\G}{\mathcal{G}}

\newcommand{\conv}{\operatorname{conv}}
\newcommand{\cvx}{\operatorname{cvx}}
\renewcommand{\Re}{\mathbb{R}}

\newcommand{\VC}{\textup{VC}}

\newcommand{\size}[1]{\left| #1 \right|}

\begin{document}

\maketitle

\begin{abstract}
Helly, Carath\'eodory, and Radon numbers encode three different kinds of finite certificates in a convexity space: for the emptiness of an intersection, for membership in a convex hull, and for the existence of intersecting hulls. We study exact versions of these certificates, in which a subfamily must preserve the whole intersection or a subset must preserve the whole hull. Our first main result shows that, for finite configurations in an arbitrary convexity space, five a priori different boundedness conditions are equivalent: VC-dimension, strong Helly number, strong Carath\'eodory number, comatching number, and strong Radon number (with the expected additive-one shift). We also obtain equivalent layered Tverberg-type decompositions and colorful consequences.

The common mechanism is exposed by the bipartite incidence graph between points and a generating family. For finite spaces, the unique minimal generator yields a natural dual convexity space; we characterize double dualization and prove that the strong parameters are duality invariant. The same model gives a polynomial-size, $O(t^4)$, realization of Bukh's counterexample to the Calder--Eckhoff partition conjecture.

Finally, we obtain the first Tverberg bound for separable convexity
spaces that is simultaneously linear in the number of parts and
polynomial in the Radon number. If an $S_3$-separable convexity space
has Helly number $h$ and its halfspaces have VC-dimension $d$, then
$r_t=O(dh\log h)\,t$; in particular, Radon number $r$ gives
$r_t=O(r^2\log r)\,t$. The bound attains the weak-Eckhoff scale $O(rt)$ whenever the Helly number is bounded. In particular, for axis-parallel box convexity in $\Re^k$ our theorem gives the optimal order
$r_t=O(rt)$ uniformly in every dimension. This appears to be the
first dimension-uniform estimate of weak-Eckhoff order for box
convexity, whereas the previous direct theory was confined to
dimension three.
\end{abstract}

\section{Introduction}\label{sec:intro}

Helly's, Carath\'eodory's, and Radon's theorems are three basic finite-certificate principles in convexity. Helly's theorem gives a small certificate that an intersection is empty, Carath\'eodory's theorem gives a small certificate that a point belongs to a convex hull, and Radon's theorem forces two intersecting hulls. In Euclidean space these principles reinforce one another, but their familiar proofs use order, topology, and linear separation. A basic problem of abstract convexity is to determine which implications survive when only the closure structure remains.

A \emph{convexity space} is a pair $(X,\C)$ such that $\varnothing,X\in\C$, the family $\C$ is closed under arbitrary intersections, and the union of every inclusion chain in $\C$ again belongs to $\C$. Its members are the convex sets, and
$\conv(Y)=\bigcap\{C\in\C:Y\subseteq C\}$
denotes the convex hull of $Y\subseteq X$. This framework, introduced by Levi~\cite{levi1951helly}, includes ordinary and lattice convexity, closure systems arising from finite matroids, and many combinatorial convexities; see~\cite{barany2022helly,holmsen2024helly} for broader context.

The classical parameters measure different finite witnesses. The Helly number $h$ bounds the size of an empty-intersection certificate, the Carath\'eodory number $c$ bounds a certificate for membership in a hull, and the Radon number $r=r_2$ forces a nontrivial partition with intersecting hulls. More generally, the $t$-th Tverberg number $r_t$ is the least $m$ such that every $m$-point set has a partition into $t$ nonempty parts whose convex hulls meet. In Euclidean space,
$r_t\bigl(\cvx(\Re^d)\bigr)=(d+1)(t-1)+1$
by Tverberg's theorem~\cite{tverberg1966generalization}. Calder~\cite{calder1971some} and Eckhoff~\cite{eckhoff1979radon} conjectured that the same formula, written as
\[
r_t\le (r-1)(t-1)+1,
\]
should follow from the Radon number alone in every convexity space. Bukh disproved this by constructing, for every $t\ge3$, a space with $r=4$ and $r_t\ge3t-1$~\cite{bukh2010radon}. Thus the ordinary Radon number does not determine the sharp Tverberg constant.

The failure of the exact partition conjecture leaves open the \emph{weak Eckhoff conjecture}, which asks whether
\[
r_t=O(rt)
\]
holds in every convexity space. This is the smallest possible order of magnitude even within the $S_4$-separable class: Euclidean convexity has $r=d+2$ and attains $(r-1)(t-1)+1$. P\'alv\"olgyi proved that $r_t$ is linear in $t$ for every fixed $r$, but the known general dependence on $r$ is enormous~\cite{palvolgyi2022radon}.

Bukh's obstruction suggests two complementary ways forward. One may strengthen the certificates themselves, requiring a small subfamily to preserve an entire intersection or a small subset to preserve an entire convex hull. Alternatively, one may retain the ordinary Radon number but impose a weak geometric separation axiom. We pursue both directions.

\subsection{Main results}\label{subsec:intro:contributions}
Ordinary Helly theory records whether an intersection is empty; in several important settings one can preserve the entire intersection. For example, Alon, Jin, and Sudakov proved that every finite family of Hamming balls of a fixed radius $q$ has a subfamily of at most $2^{q+1}$ members with exactly the same intersection~\cite{alon2024helly}. Similar exact-intersection statements occur in algebraic, model-theoretic, and discrete-geometric settings, but their common structure is not apparent from the usual Helly formulation.

We isolate this structure through the following parameters. The \emph{strong Helly number} is the least $d$ such that every finite $\F\subseteq\C$ contains $\F'\subseteq\F$ with $|\F'|\le d$ and $\bigcap\F'=\bigcap\F$. Dually, the \emph{strong Carath\'eodory number} is the least $d$ such that every finite $Y\subseteq X$ contains $Y'\subseteq Y$ with $|Y'|\le d$ and $\conv(Y')=\conv(Y)$. The \emph{strong Radon number} is the least $m$ such that every finite $S\subseteq X$ with $|S|\ge m$ has a nontrivial partition $S=S_1\mathbin{\dot\cup}S_2$ satisfying $\conv(S_1)\subseteq\conv(S_2)$.

A fourth parameter is the \emph{comatching number}. For a set system $\F\subseteq2^X$, it is the largest $m$ for which there are distinct points $p_1,\ldots,p_m$ and sets $C_1,\ldots,C_m\in\F$ such that $p_i\in C_j$ exactly when $i\ne j$. This configuration appears implicitly in many exact Helly arguments, including Deza and Frankl's work on bounded-degree hypersurfaces~\cite{deza1987helly}, and was recently named and used systematically by Pohoata, Yang, and Zhang~\cite{pohoata2025colorful}. Finally, $\VC(\C)$ is the largest size of a set shattered by the traces of $\C$.

Our first main theorem identifies all these notions.

\begin{thm}\label{thm:main:short}
Let $(X,\C)$ be a convexity space and let $d\ge1$. The following are equivalent:
\begin{enumerate}
    \item $\VC(\C)\le d$;
    \item the strong Helly number is at most $d$;
    \item the strong Carath\'eodory number is at most $d$;
    \item the comatching number is at most $d$;
    \item the strong Radon number is at most $d+1$.
\end{enumerate}
\end{thm}

Several individual implications were previously known, sometimes in different terminology and sometimes only for particular classes of convexity spaces; we review them in Subsection~\ref{subsec:strong:history}. The novelty is the complete identification of the parameters, including independence from the chosen generating family, together with two-sided layered decompositions and the resulting duality and colorful consequences. The full statement, Theorem~\ref{thm:strongEquivalences}, in particular gives layered strong Tverberg decompositions on both the point side and the convex-set side.

Theorem~\ref{thm:main:short} is also a transfer principle: a bound proved in any one of its five languages immediately yields all the others. For the convexity generated by Hamming balls of radius $q$ (in ambient dimension greater than $q$ and over an alphabet of size at least two), the theorem of Alon, Jin, and Sudakov implies that the common parameter is $2^{q+1}$; hence the strong Carath\'eodory and comatching numbers are $2^{q+1}$, the strong Radon number is $2^{q+1}+1$, and the layered and colorful consequences below hold with the same dimension-free parameter. For the closure convexity of a matroid of rank $\rho$, the common parameter is exactly $\rho$. For axis-parallel box convexity in $\Re^k$, it is $2k$: an intersection of boxes is determined by at most one extremal constraint in each coordinate direction, and a box hull is determined by at most one point realizing each coordinate extreme. Likewise, the comatching bounds of Deza and Frankl for bounded-degree projective hypersurfaces automatically yield exact Helly, Carath\'eodory, and Radon statements for the generated convexity. These examples show that the common invariant captures coding-theoretic radius, matroid rank, coordinate complexity, and algebraic dimension within a single framework.

The apparent asymmetry is that the Helly property is a meet statement about convex sets or generators, whereas the Carath\'eodory property is a closure statement about points. Neither the symmetry between these two sides nor independence from the generating family is visible from the definitions. We introduce a bipartite incidence model which removes this asymmetry: intersections and convex hulls become the two common-neighborhood operations. Consequently, all the equivalences reduce to the exclusion of a single induced comatching configuration rather than requiring a collection of unrelated implications.

The equivalence also has immediate colorful consequences. As a representative application, we obtain the following strong colorful Tverberg theorem.

\begin{thm}[Colorful Tverberg theorem]
Let $d,t\ge1$, let $(X,\C)$ satisfy $\VC(\C)\le d$, and let $S_1,\ldots,S_{d+1}\subseteq X$ be pairwise disjoint finite sets. If
$|S_i|\ge d(t-1)+1$
for every $i\in[d+1]$, then there exist pairwise disjoint transversals $T_1,\ldots,T_t$ of the color classes whose convex hulls have a common point; that is,
\[
\bigcap_{j=1}^t\conv(T_j)\ne\varnothing.
\]
\end{thm}

We also compare the strong Helly number with two refinements from the literature. The \emph{breadth} asks for exact intersection certificates only when the total intersection is nonempty, while the \emph{comatching number with intersection} requires one additional point common to all sets in the comatching. We prove that these independently introduced parameters coincide and use this identification to sharpen a strong colorful Helly statement of Pohoata, Yang, and Zhang~\cite{aschenbrenner2016vapnik,pohoata2025colorful}.

\medskip
\noindent\textbf{Incidence models and duality.}
A subfamily $\G\subseteq\C$ is a generator if every convex set is an intersection of members of $\G$. The bipartite incidence graph between $X$ and $\G$ determines the closure operator through common neighborhoods. For finite convexity spaces, the minimal generator is unique, so transposing the two sides produces a natural dual space. We characterize when double dualization recovers the original space and prove that the equivalent strong parameters are invariant under this duality. This connects exact Helly certificates and exact hull certificates at the structural level, rather than merely through numerical inequalities.

The same model gives a shorter and substantially smaller realization of Bukh's counterexample. We encode the required local incidence pattern directly with only $O(t^4)$ points. Besides replacing the large auxiliary realization by a polynomial-size one, the model separates the two mechanisms in the construction: local witnesses that force every four-point set to be Radon, and a global configuration that obstructs a Tverberg $t$-partition.

\medskip
\noindent\textbf{Tverberg numbers under weak separation.}
We next follow the second route suggested above. A \emph{halfspace} is a convex set whose complement is also convex. Following van de Vel~\cite{van1993theory}, an $S_3$-separable space allows separation of a point from a convex set by complementary halfspaces, whereas $S_4$-separability allows separation of two disjoint convex sets. Recent work under the stronger $S_4$ axiom gives
$r_t=O(rt^2\log t)$
by Alon and Smorodinsky~\cite{AlonSmorodinsky}, and Keller and Smorodinsky subsequently obtained
$r_t=O(r^2t\log t)$
for $t>r$~\cite{keller2026colorful}. We obtain a genuinely linear dependence on the number of parts under the weaker $S_3$ axiom.

\begin{thm}\label{thm:separable}
Let $(X,\C)$ be an $S_3$-separable convexity space with Helly number $h$, and let $\B$ be its family of halfspaces. If $\VC(\B)\le d$, then
\[
r_t=O\bigl(dh\log h\bigr)t.
\]
In particular, if the Radon number is $r$, then
\[
r_t=O(r^2\log r)t.
\]
\end{thm}

The proof combines a Helly-type centerpoint with a simultaneous packing of disjoint $\varepsilon$-nets. The centerpoint lies in the hull of every sufficiently large subset of the original configuration. A random equipartition then produces linearly many pairwise disjoint nets at once, and $S_3$-separation converts a failure of hull containment into a complementary halfspace missed by one of the nets. Choosing all nets simultaneously avoids the loss incurred by successive deletion.

The refined bound in Theorem~\ref{thm:separable} is substantially
stronger than its formulation in terms of the Radon number alone. If
the Helly number is bounded, then $d\le r-1$ gives
\[
r_t=O(rt),
\]
which attains the weak-Eckhoff scale. A basic example is
axis-parallel box convexity in $\Re^k$: it is $S_4$-separable, has
Helly number $2$, and has Radon number $\Theta(\log k)$; see
\cite{Eckhoff2001, holmsen2024helly}. Theorem~\ref{thm:separable} therefore gives
$r_t=O(rt)$ for this entire family, in every dimension $k$, despite the
unbounded Radon number. This is worth contrasting with the direct
approach. Eckhoff~\cite{Eckhoff2001} determined the Tverberg number for box convexity
up to small constants only in $\Re^3$, where the Radon number is a fixed
constant, showing $\lceil 5t/2\rceil\le r_t\le\lceil 18t/7\rceil$; he
notes that his permutation-based methods do not appear to extend to
higher dimensions, and that for $k\ge3$ the problem is wide open, with
not even a conjectured form for the answer. In $\Re^3$ his constants are
of course sharper than ours; the point is that no comparable bound was
previously available in any higher dimension. Our result bypasses this
difficulty entirely: the bounded Helly number alone forces the
weak-Eckhoff scale $r_t=O(rt)$ uniformly in $k$, without any
dimension-specific analysis.

At the opposite extreme, bounded halfspace VC-dimension gives
\[
r_t=O(r\log r)t.
\]
A natural example in this regime is the geodesic convexity of the
Johnson graph $J(n,2)$. More precisely, let
$X_n=\binom{[n]}2,$
join two members of $X_n$ when they intersect in one element, and let
$\C_n$ consist of the geodesically convex vertex sets of the resulting
graph: a set belongs to $\C_n$ if it contains every shortest path
between each pair of its vertices. The graph $J(n,2)$ is the basis
graph of the uniform matroid $U_{2,n}$ and, equivalently, the
$1$-skeleton of the hypersimplex
$\Delta(n,2)
 =
\conv\left\{\mathbf 1_e:e\in\binom{[n]}2\right\}.$
Basis graphs of matroids are $S_3$-separable for geodesic convexity
\cite{chepoi2024separation}. For $n\ge5$, this space is not
$S_4$-separable. An elementary calculation shows that the nontrivial halfspaces of
$(X_n,\C_n)$ are precisely the coordinate stars
$H_v^+
 =
\left\{
e\in\binom{[n]}2:v\in e
\right\}$
and their complements $H_v^-=X_n\setminus H_v^+$. Consequently,
$\VC\bigl(\{H_v^+,H_v^-:v\in[n]\}\bigr)=3,$
$h(X_n,\C_n)=n-1,$
and 
$r(X_n,\C_n)=n.$
Theorem~\ref{thm:separable} therefore gives
\[
r_t(X_n,\C_n)
 =
O(n\log n)t
 =
O(r\log r)t.
\]
For this particular space an even sharper estimate follows from its
additional exact-certificate structure. Its rank, equivalently
$\VC(\C_n)$, equals $n-1$, so Jamison's rank theorem
\cite{jamison1981partition} gives
\[
r_t(X_n,\C_n)\le (n-1)(t-1)+1.
\]
Thus the general $S_3$ theorem recovers the optimal linear scale for
this natural polytopal example up to one logarithmic factor using only
the Helly number and the much simpler halfspace system.

There are also natural examples in which neither of the two parameters
in Theorem~\ref{thm:separable} is bounded. For lattice convexity on
$\mathbb Z^k$, whose convex sets are the sets
$K\cap\mathbb Z^k$ with $K\subseteq\Re^k$ convex, the halfspaces have
VC-dimension $k+1$ and the Helly number is $2^k$; for $k\ge3$, the
Radon number satisfies
$5\cdot2^{k-2}+1\le r(\mathbb Z^k)\le k(2^k-1)+3$ by results of
Sierksma and Onn (see~\cite{holmsen2024helly}). Theorem~\ref{thm:separable}
thus gives $r_t(\mathbb Z^k)=O(k^2 2^k)\,t=O(r(\log r)^2)\,t$,
recovering the scale of the best lattice-specific
bound~\cite{deLoeraEtAl2017QuantitativeTverberg},
$r_t(\mathbb Z^k)\le k2^k(t-1)+1=O(r\log r)\,t$, to within one
logarithmic factor and without any arithmetic geometry. Unlike the
Johnson-graph example, lattice convexity in dimension at least two has
unbounded rank, so this bound follows neither from the strong-invariant
theorem nor from Jamison's finite-rank bound --- a genuine application
of the separable theory beyond the exact-certificate regime.

\medskip
\noindent\textbf{Relation to colorful VC methods and further consequences.}
Keller and Smorodinsky~\cite{keller2026colorful} recently introduced a
colorful $k$-wise extension of VC-dimension and used it to obtain
Tverberg-type theorems for separable convexity spaces. Suppose, in
addition to $S_3$-separability,
that the space is $S_4$-separable, and write $r$ for its Radon number
and $h$ for its Helly number. Their $k$-wise Tverberg theorem, applied
with $k=\min\{h,t\}$, gives
\[
 r_t=O\bigl(r\min\{h,t\}\,t\log(2rt)\bigr),
\]
whereas Theorem~\ref{thm:separable} gives
\[
 r_t=O\bigl(rh\log h\bigr)t.
\]
Thus the colorful $k$-wise method is particularly effective when
$t<h$, while our centerpoint--net method is stronger in the
weak-Eckhoff regime $t\ge h$: it removes the logarithmic dependence on
$t$ and, more importantly, requires only $S_3$-separation.

The proof of Theorem~\ref{thm:separable} also strengthens the chain of
selection and piercing consequences developed in
\cite{keller2026colorful}. Put
 $a=O\bigl(dh\log h\bigr),$
where $d$ is the VC-dimension of the halfspaces. We prove both
uncolored and colorful selection lemmas with parameter $a$ and an
absolute positive selection density. It follows that every finite
point set has weak $\varepsilon$-nets of size
\[
 2\left(\frac e\varepsilon\right)^a,
\]
and that the corresponding quantitative $(p,q)$-theorem has exponent
$a$. The colorful selection lemma also yields a second Tverberg
tradeoff with $a$ color classes, each of size $O(at)$. In terms of the
Radon number, $a=O(r^2\log r)$. Under $S_4$-separation this improves
the $O(r^3)$ parameter in the corresponding results of Keller and
Smorodinsky; under $S_3$-separation it improves their
$O(r^4\log r)$ parameter and removes the auxiliary compactness
assumption used in their weak-separation adaptation. Full statements
and proofs are given in
Appendix~\ref{subsec:separable:consequences}.

\medskip
\noindent\textbf{Organization.}
Section~\ref{sec:bipartite} develops the incidence model and duality. Section~\ref{sec:strong} proves the equivalence theorem and its colorful consequences. Section~\ref{sec:Bukh} gives the polynomial-size version of Bukh's construction. Section~\ref{sec:separable} proves the Tverberg bound under $S_3$-separability and discusses the examples and relations above, and Section~\ref{sec:discussion} records the principal open directions.
\section{The bipartite graph model, and duality}\label{sec:bipartite}

For families $\F,\G \subseteq 2^X$ of subsets of a base set $X$, let $\bigcap \F = \bigcap_{C\in \F} C$ and $\G^\cap = \{\bigcap \F: \F \subseteq \G\}$.
We also use the convention $\bigcap \varnothing = X$, whenever the base set is clear from context.

For a convexity space $(X,\C)$, we call a subfamily $\G\subseteq\C$ a generator family if $\G^\cap=\C$. Let $I(X,\G)$ be the bipartite incidence graph with vertex classes $X$ and $\G$, where $x$ is adjacent to $C$ exactly when $x\in C$. We call $I(X,\G)$ a bipartite model of the space. Different generator families may give different models. Conversely, every bipartite graph determines an intersection-closed set system on either vertex class; in the finite case this is automatically a convexity space, while in the infinite case closure under unions of chains must be checked separately.

A basis $\B$ is an inclusion-wise minimal generator family.

\begin{claim}\label{cl:uniqueBasis}
    Every finite convexity space has a unique basis.
\end{claim}

\begin{proof}
    Let $(X,\C)$ be a finite convexity space. As $\C$ is a finite generator family, it has a minimal subfamily which is a basis.

    Suppose for contradiction that we have two bases $\B_1$ and $\B_2$, and let $K \in \B_1 \setminus \B_2$.  Let $\B_2' \subseteq \B_2$ be a subfamily with $\bigcap \B_2' = K$.  Let $\B_1' \subseteq \B_1$ be a minimal subfamily such that $\B_2' \subseteq \B_1'^\cap$. Then $K \in \B_1'^\cap$, but $\B_1' \subseteq \B_1 \setminus \{K\}$ as every member of $\B_2'$ strictly contains $K$. This contradicts the minimality of $\B_1$.
\end{proof}

We call $I(X,\B)$ the \emph{minimal} bipartite graph model of a finite convexity space.

\subsection{Duality of convexity spaces}\label{subsec:bipartite:duality}

Given a finite convexity space $(X,\C)$ with basis $\B$, let $X^\top$ be a labeled copy of $\B$, writing $g(B)\in X^\top$ for the copy of $B\in\B$. For each $x\in X$, introduce a labeled generator
\[
G_x^\top=\{g(B):B\in\B,\ x\in B\}\subseteq X^\top,
\]
and let $\C^\top$ consist of all intersections of the indexed family $\G^\top=(G_x^\top)_{x\in X}$. We call $(X^\top,\C^\top)$ the dual convexity space. Keeping the generators indexed is important when two points have the same neighborhood. Its bipartite model is obtained from $I(X,\B)$ by interchanging the two vertex classes.

As an example, if $X$ is the point set of a finite projective space and $\B$ is its family of hyperplanes, then $\C$ is the family of projective subspaces. The relation $x \in B \iff g(B) \in G_x^\top$ is exactly the usual point--hyperplane incidence relation under projective duality.

In the bipartite graph model, convex hulls of points and intersections of convex sets both correspond to certain common neighborhoods. For a graph $(V,E)$ and a subset $S \subseteq V$ of the vertices, let $N(S) = \{v \in V : \forall s \in S: \{v,s\} \in E\}$ be the set of common neighbors of $S$. For a convexity space $(X,\C)$ with generator set $\G$, we have $\bigcap \F = N(\F)$ for any $\F \subseteq \G$ and $\conv(Y) = N(N(Y))$ for any $Y \subseteq X$, where the neighborhoods are with respect to the bipartite graph model. We make the convention that for $Y \subseteq X$, $N(N(Y)) = X$ if $N(Y) = \varnothing$.

For future use, we mention the following property of the common neighborhood operation.
\begin{claim}\label{claim:NNN=N}
    For every $A \subseteq X$, we have $N(N(N(A))) = N(A)$.
\end{claim}
\begin{proof}[Proof of Claim~\ref{claim:NNN=N}]
    Recall that $A \subseteq \conv(A) = N(N(A))$.
    Applying this to $N(A)$, we get $N(A) \subseteq N(N(N(A)))$.

    On the other hand, note that $A \subseteq B \subseteq X$ implies $N(A) \supseteq N(B)$.
    Therefore $N(A) \supseteq N(N(N(A)))$.
    \renewcommand\qedsymbol{$\blacksquare$}
\end{proof}

We say two convexity spaces $(X_1, \C_1)$ and $(X_2, \C_2)$ are \emph{isomorphic} if there exists a bijection $f: X_1 \to X_2$ such that $C \in \C_1$ if and only if $f(C) \in \C_2$.
The transposed model need not be minimal, so double dualization may identify redundant generators. The exact criterion is the following.

\begin{claim}\label{cl:doubledual}
We have $((X^\top)^\top,(\C^\top)^\top)\simeq(X,\C)$ if and only if there do not exist $p\in X$ and $Y\subseteq X\setminus\{p\}$ with
\[
\conv(\{p\})=\conv(Y).
\]
\end{claim}

\begin{proof}
The double dual is isomorphic to $(X,\C)$ exactly when the transposed generator family $\G^\top=(G_x^\top)_{x\in X}$ is minimal. The generator indexed by $p$ is redundant precisely when
\[
G_p^\top=\bigcap_{y\in Y}G_y^\top
\]
for some $Y\subseteq X\setminus\{p\}$. In the incidence graph this is equivalent to $N(p)=N(Y)$. By Claim~\ref{claim:NNN=N}, this holds exactly when $N(N(p))=N(N(Y))$, that is, when $\conv(\{p\})=\conv(Y)$.
\end{proof}

In particular, the double dual is isomorphic to the original space whenever the space is point-convex, meaning that every singleton is convex.

\subsection{A generator-relative dual Radon number}\label{subsec:bipartite:dualRadon}
The transpose of a bipartite model depends on the chosen generator, so an ordinary dual Radon parameter is naturally attached to the pair $(\C,\G)$ rather than to $\C$ alone.

\begin{definition}
Let $\G$ be a generator family of proper convex sets. The \emph{dual Radon number} $r^\top_{\G}$ is the least integer $m$ such that every $m$-element family $\F\subseteq\G$ has a nontrivial partition $\F=\F_1\mathbin{\dot\cup}\F_2$ satisfying
\[
\conv\left(\bigcap\F_1\cup\bigcap\F_2\right)\ne X.
\]
\end{definition}

If $(X,\C)$ is finite and $\G=\B$ is its basis, then $r^\top_{\B}$ is exactly the Radon number of the dual convexity space. In the transposed incidence model, the displayed condition is equivalent to
\[
\conv_{\C^\top}(\F_1)\cap\conv_{\C^\top}(\F_2)\ne\varnothing.
\]

For Euclidean convexity, let $\mathcal H_d$ be the family of all proper affine halfspaces, open or closed. This family generates $\cvx(\Re^d)$.

\begin{claim}
The dual Radon number of the halfspace model of Euclidean convexity is
\[
r^\top_{\mathcal H_d}=d+2.
\]
\end{claim}

\begin{proof}
Write $H_i=\{x:\langle a_i,x\rangle\mathrel{\triangleleft_i}b_i\}$, where $a_i\ne0$ and $\triangleleft_i$ is either $<$ or $\le$. By the conic Radon theorem, any $d+2$ normals admit a nontrivial partition $I\mathbin{\dot\cup}J$ for which
\[
\operatorname{cone}\{a_i:i\in I\}\cap
\operatorname{cone}\{a_j:j\in J\}
\]
contains a nonzero vector $u$. If both $\bigcap_{i\in I}H_i$ and $\bigcap_{j\in J}H_j$ are nonempty, then the functional $x\mapsto\langle u,x\rangle$ is bounded above on each intersection, so their union lies in a proper halfspace. If one intersection is empty, the conclusion is immediate. This proves the upper bound.

For the lower bound, choose $d+1$ vectors $a_1,\ldots,a_{d+1}$ spanning $\Re^d$ whose unique linear dependence has all coefficients positive, and put $H_i=\{x:\langle a_i,x\rangle\le1\}$. For every nontrivial partition $I\mathbin{\dot\cup}J$, the cones generated by the two corresponding sets of normals meet only at the origin. By the recession-cone form of Farkas' lemma, no nonzero linear functional is bounded above on both $\bigcap_{i\in I}H_i$ and $\bigcap_{j\in J}H_j$. Hence the convex hull of their union is all of $\Re^d$, so these $d+1$ halfspaces have no dual Radon partition.
\end{proof}

Unlike the strong parameters studied below, the ordinary Radon number is not controlled by the basis-dual parameter.

\begin{claim}\label{cl:noDualRadonBound}
There is no function $f$ such that $r(\C)\le f(r^\top_{\B})$ for every finite convexity space with basis $\B$.
\end{claim}

\begin{proof}
For finite convexity spaces $(X_1,\C_1)$ and $(X_2,\C_2)$, define
\[
(X_1,\C_1)\sqcup(X_2,\C_2)
=\bigl(X_1\sqcup X_2,\,\C_1\cup\C_2\cup\{X_1\sqcup X_2\}\bigr).
\]
Its Radon number is $\max\{r(\C_1),r(\C_2)\}$. Its basis consists of the two original bases together with $X_1$ and $X_2$. The pair $\{X_1,X_2\}$ has no dual Radon partition, whereas every three basis elements do: two belong to the same component, and grouping the third with one of them makes one of the two intersections empty or keeps both intersections in that component. Thus the basis-dual Radon number is $3$. Taking one component with arbitrarily large Radon number proves the claim.
\end{proof}

\section{Equivalence of strong invariants}\label{sec:strong}
In this section we state several invariants of convexity spaces which turn out to be equivalent, thus characterizing a subclass of convexity spaces.
These invariants are all unbounded for the Euclidean space, preserved under duality, and most of them are strengthenings of the classical invariants. Some of the implications between the items below already explicitly or implicitly appear in literature in various guises; we list them in the next subsection.

\begin{definition}
    A subset $Y\subseteq X$ of points in a convexity space $(X,\C)$ is in convex position, if for all $y \in Y$ we have $y \not\in \conv(Y\setminus\{y\})$.
\end{definition}

\begin{definition}
Let $\F\subseteq2^X$ be a set system. Its \emph{comatching number} is the largest $m$ for which there are distinct points $p_1,\ldots,p_m\in X$ and sets $C_1,\ldots,C_m\in\F$ such that
\[
p_i\in C_j\quad\Longleftrightarrow\quad i\ne j.
\]
Equivalently, the incidence graph of $(X,\F)$ contains the bipartite complement of a matching of size $m$ as an induced subgraph.
\end{definition}

\begin{observation}\label{obs:generator-comatching}
If $\G$ generates $\C$, then $\G$ and $\C$ have the same comatching number.
\end{observation}

\begin{proof}
Only one direction requires proof. Suppose that $p_1,\ldots,p_m$ and $C_1,\ldots,C_m\in\C$ form a comatching. Write each $C_i$ as an intersection of generators. Since $p_i\notin C_i$, some generator $G_i$ in that representation excludes $p_i$; every $p_j$ with $j\ne i$ lies in $C_i$ and hence in $G_i$. Thus the points $p_i$ and generators $G_i$ form a comatching in $\G$.
\end{proof}

\begin{thm}\label{thm:strongEquivalences}
Let $(X,\C)$ be a convexity space with generator family $\G\subseteq\C$. The following are equivalent.
\begin{enumerate}
    \item\label{prop:VCofC} $\VC(\C)\le d$.
    \item\label{prop:rank} Every subset of $X$ in convex position has size at most $d$.
    \item\label{prop:comatching} The comatching number of $\C$ is at most $d$.
    \item\label{prop:strongHellyC} Every finite $\F\subseteq\C$ has a subfamily $\F'\subseteq\F$ of size at most $d$ with $\bigcap\F'=\bigcap\F$.
    \item\label{prop:strongHellyB} Every finite $\F\subseteq\G$ has a subfamily $\F'\subseteq\F$ of size at most $d$ with $\bigcap\F'=\bigcap\F$.
    \item\label{prop:strongRadon} Every $(d+1)$-element set $S\subseteq X$ has a nontrivial partition $S=S_1\mathbin{\dot\cup}S_2$ such that $\conv(S_1)\subseteq\conv(S_2)$.
    \item\label{prop:strongCaratheodory} Every finite $S\subseteq X$ has a subset $S'\subseteq S$ of size at most $d$ with $\conv(S')=\conv(S)$.
    \item\label{prop:layerStructureX} Every finite $Y\subseteq X$ has a partition into nonempty blocks $Y=Y_1\mathbin{\dot\cup}\cdots\mathbin{\dot\cup}Y_t$, each of size at most $d$, such that
    \[
    \conv(Y_1)\subseteq\cdots\subseteq\conv(Y_t).
    \]
    \item\label{prop:layerStructureB} Every finite $\F\subseteq\G$ has a partition into nonempty blocks $\F=\F_1\mathbin{\dot\cup}\cdots\mathbin{\dot\cup}\F_t$, each of size at most $d$, such that
    \[
    \bigcap\F_1\subseteq\cdots\subseteq\bigcap\F_t.
    \]
    \item\label{prop:layerStructureC} The same statement as in \ref{prop:layerStructureB} holds for every finite $\F\subseteq\C$.
\end{enumerate}
\end{thm}

\begin{proof}
A finite set $S$ is in convex position exactly when $\conv(A)\cap S=A$ for every $A\subseteq S$, which is equivalent to $S$ being shattered by $\C$. This proves \ref{prop:VCofC}$\Leftrightarrow$\ref{prop:rank}.

A comatching $p_1,\ldots,p_m;C_1,\ldots,C_m$ places the points in convex position. Conversely, if $p_1,\ldots,p_m$ are in convex position, then $C_i=\conv(\{p_j:j\ne i\})$ form a comatching. Hence \ref{prop:rank}$\Leftrightarrow$\ref{prop:comatching}.

For an arbitrary set system $\F$, its comatching number is at most $d$ exactly when every finite subfamily has an exact intersection certificate of size at most $d$. Indeed, an inclusion-minimal subfamily $C_1,\ldots,C_m$ with a prescribed intersection yields points
\[
p_i\in\bigcap_{j\ne i}C_j\setminus C_i,
\]
and hence a comatching of size $m$; conversely, the sets in a comatching form an inclusion-minimal intersection representation. Applying this to $\C$ and to $\G$, together with Observation~\ref{obs:generator-comatching}, proves the equivalence of \ref{prop:comatching}, \ref{prop:strongHellyC}, and \ref{prop:strongHellyB} without any finiteness assumption on the chosen generator representations.

If \ref{prop:rank} holds and $|S|=d+1$, some $p\in S$ lies in $\conv(S\setminus\{p\})$; the partition $\{p\}\mathbin{\dot\cup}(S\setminus\{p\})$ proves \ref{prop:strongRadon}. Conversely, a nontrivial partition with $\conv(S_1)\subseteq\conv(S_2)$ places every point of $S_1$ in the hull of the remaining points, so $S$ is not in convex position. Thus \ref{prop:rank}$\Leftrightarrow$\ref{prop:strongRadon}.

Under \ref{prop:rank}, repeatedly delete a point $p$ satisfying $p\in\conv(S\setminus\{p\})$. Each deletion preserves the hull, and the process ends with a set in convex position, hence with at most $d$ points. This proves \ref{prop:strongCaratheodory}. The converse follows because a set in convex position has no proper subset with the same hull.

Finally, repeatedly apply the appropriate exact Helly statement to the remaining family. The selected blocks have size at most $d$, and their intersections form an increasing chain. Repeatedly applying exact Carath\'eodory to the remaining point set gives blocks with decreasing hulls; reverse their order to obtain \ref{prop:layerStructureX}. Conversely, the first block in either Helly chain has the intersection of the whole family, and the last block in the Carath\'eodory chain has the hull of the whole set. This proves the three layered equivalences.
\end{proof}

\begin{rem}
    The finiteness assumptions on families and subsets are necessary for the equivalences to hold. As an example, consider $X = \mathbb{N}$ where convex sets are all initial segments $\{1, \dots, n\}$ along with $\varnothing$ and $\mathbb{N}$. The maximum size of a set in convex position is $d=1$ (it has rank 1). However, the infinite set $S = \mathbb{N}$ has $\conv(S) = \mathbb{N}$, but no finite subset $S'$ has $\conv(S') = \mathbb{N}$. This breaks Strong Carath\'eodory (Prop~\ref{prop:strongCaratheodory}) for infinite sets.
\end{rem}

If the space is finite, the strong invariant is preserved by duality.

\begin{thm}\label{thm:VC-dual}
The VC-dimension of a finite convexity space equals the VC-dimension of its dual.
\end{thm}

\begin{proof}
Let $\B$ be the basis of $(X,\C)$ and $\G^\top$ the transposed generator family. By Observation~\ref{obs:generator-comatching}, the comatching numbers of $\C$ and $\B$ agree, as do those of $\C^\top$ and $\G^\top$. The incidence graphs $I(X,\B)$ and $I(X^\top,\G^\top)$ differ only by interchanging their two vertex classes, so their largest induced bipartite complements of matchings have the same size. The result follows from Theorem~\ref{thm:strongEquivalences}.
\end{proof}

\subsection{Previous results and new examples}\label{subsec:strong:history}

Below we list previous results proving some of the implications in Theorem~\ref{thm:strongEquivalences} or boundedness of the strong invariants in different convexity spaces. The definitions of all the different spaces can be found in the corresponding papers.

\begin{enumerate}

\item Jamison calls the maximum size $d$ of a subset in convex position the \emph{rank} of the convexity space, proves in \cite{jamison1981partition} that (\ref{prop:rank}) implies (\ref{prop:layerStructureX}) in Theorem~\ref{thm:strongEquivalences}, and states the following weaker form.

\begin{thm}[Proposition 3 of \cite{jamison1981partition}]
    We have $r_t \leq d(t-1) + 1$.
\end{thm}

\item Deza and Frankl proved that the comatching number of hypersurfaces of degree at most $D$ in a projective space of dimension $d$ is at most $\binom{d+D}{d}$, and showed that it implies the same bound on their (strong) Helly number~\cite{deza1987helly}. The term comatching number was first used by Pohoata, Yang and Zhang in~\cite{pohoata2025colorful}, where they proved an optimal bound on the colorful Helly number for convexity spaces of bounded comatching number.

\item In \cite{gartner2014online}, G\"artner and Missura mention the equivalence of strong Helly number and VC-dimension and that Radon number is bounded by the VC-dimension plus one.
As a consequence, they observe the following.

\begin{thm}[Theorem 1 of G\"artner and Missura \cite{gartner2014online}]
    If $\C$ has VC-dimension $d$, then
    \[
    \left|\bigcap\{C \in \C: |C| \geq \varepsilon |X| \}\right| \geq |X| - d|X|(1-\varepsilon).
    \]
\end{thm}

\item Chernikov and Mennen prove strong Radon, VC dimension, strong Carath\'eodory and strong Tverberg for nonarchimedean convex sets in \cite{chernikov2023combinatorial}. They do not state implications between the invariants, but the proofs of the other three invariants are all based on the strong Radon property.

\item Alon, Jin and Sudakov \cite{alon2024helly} prove strong Helly for Hamming balls and mention that it can be viewed as Radon's theorem.
In our language, they prove (\ref{prop:strongHellyB}) implies (\ref{prop:strongRadon}).

\item Rao \cite{rao2025helly} bounds the Radon number for $d$-intervals, but from the proof the same bound on the strong Radon number can also be reconstructed.

\item The strong Helly number for $H$-convex sets was determined in~\cite{frankl2025helly}. The bound applies for example to axis parallel boxes, where the strong (colorful) Helly number is $2d$.
\end{enumerate}

Below we give some new examples of convexity spaces with bounded strong Helly number.

\begin{enumerate}
    \item If $X$ is the vertex set of a tree with $d$ leaves, and $\C$ is the set of all the subsets spanning connected subgraphs, then it has VC-dimension $d$. As a variant, let $X$ be any tree, and let $\C$ be the set of all subtrees with at most $d-1$ leaves, plus $X$ itself.

    \item If $X$ is the vertex set of a triangle-free graph and $\G$ is the set of edges, then the generated convexity space has VC-dimension at most $2$.

    \item If $X$ is the ground set of a matroid, and $\C$ consists of the closed sets of the matroid, then the VC-dimension of convex sets is again the rank of the matroid.
\end{enumerate}

\subsection{Breadth and comatching numbers}
We review two other related strong Helly-type parameters that are defined recently.
The first one is \emph{breadth} introduced by Aschenbrenner et al.~\cite{aschenbrenner2016vapnik}, and the other one is \emph{comatching number with intersection} introduced by Pohoata et al.~\cite{pohoata2025colorful}.
We show that these two independently defined parameters are indeed identical.

We start with definitions of Helly-type parameters of our concern.
\begin{definition}
    Let $(X, \C)$ be a convexity space and $\F, \F_1, \F_2, \ldots \subseteq \C$ be families of convex sets.
    Throughout this definition, all families are finite.
    Parameters $h, ch, sh, sch, b, \tau'$ are defined as follows.
    \begin{itemize}
        \item $h$: \emph{Helly number}; smallest integer $k \geq 1$ such that every $\F$ with $\bigcap \F = \varnothing$ has a subfamily $\F'$ of size at most $k$ such that $\bigcap \F' = \varnothing$.
        \item $ch$: \emph{colorful Helly number}; smallest integer $k \geq 1$ such that every $\F_1, \ldots, \F_k$ with $\bigcap \F_i = \varnothing$ for every $i \in [k]$ has a transversal $T = \{C_1, \ldots, C_k\}$ such that $\bigcap T = \varnothing$.
        \item $sh$: \emph{strong Helly number}; smallest integer $k \geq 1$ such that every $\F$ has a subfamily $\F'$ of size at most $k$ such that $\bigcap \F = \bigcap \F'$.
        \item $sch$: \emph{strong colorful Helly number}; smallest integer $k \geq 1$ such that every $\F_1, \ldots, \F_k$ has a transversal $T = \{C_1, \ldots, C_k\}$ such that $\bigcap T \subseteq \bigcap \F_i$ for some $i \in [k]$.
        \item $b$: \emph{breadth}; smallest integer $k \geq 1$ such that whenever $\bigcap \F \neq \varnothing$, $\F$ has a subfamily $\F'$ of size at most $k$ such that $\bigcap \F = \bigcap \F'$.
        \item $\tau'$: \emph{comatching number with (nonempty) intersection}; maximal $k$ such that there exists $C_1, \ldots,$ $C_k \in \C$ and $x_1, \ldots, x_{k+1} \in X$ (which are necessarily all distinct) such that $x_i \in C_j$ if and only if $i \neq j$.
        In other words, $(C_1, x_1), \ldots, (C_k, x_k)$ is a comatching such that $\bigcap_{i = 1}^k C_i \neq \varnothing$.
    \end{itemize}
\end{definition}
\begin{thm}
    $b = \tau'$.
\end{thm}
\begin{proof}
    ($b \leq \tau'$)
    Let $C_1, \ldots, C_m \in \C$ be convex sets guaranteeing the breadth $b$ from below; this means $\varnothing \neq \bigcap_{i=1}^m C_i$ and $\bigcap_{i=1}^m C_i \subsetneq \bigcap_{i \in I} C_i$ for every $I \in \binom{[m] }{b-1}$.
    By considering minimal family $\{C_1, \ldots, C_m\}$ satisfying these, we get $m = b$.

    Take arbitrary $x_{b+1} \in \bigcap_{i=1}^b C_i$ and $x_j \in (\bigcap_{i \neq j} C_i) \setminus (\bigcap_{i=1}^b C_i)$ for each $j \in [b]$.
    These points together with $C_1, \ldots, C_b$ form a comatching with intersection of size $b$.

    ($b \geq \tau'$)
    Let $C_1, \ldots, C_{\tau'} \in \C$ and $x_1, \ldots, x_{\tau'+1} \in X$ form a comatching with intersection.
    Then $x_{\tau'+1} \in \bigcap_{i=1}^{\tau'}C_i$ and $x_j \in (\bigcap_{i \neq j}C_i) \setminus (\bigcap_{i=1}^{\tau'}C_i)$ for every $j \in [\tau']$.
    Thus the intersection of family $\{C_1, \ldots, C_{\tau'}\}$ is nonempty and the intersection of every proper subfamily contains the original intersection properly.
    This shows $b > \tau' - 1$.
\end{proof}
Thus we actually have five parameters $h, ch, sh, sch, b$ of $(X, \C)$.
No pairs of these are the same parameter, and we investigate possible relations between them.
Here are the basic ones that are previously known or direct from definitions.
\begin{thm}\label{t:parameters}\phantom{}
    \begin{enumerate}
        \item $h \leq ch, sh \leq sch$.
        \item $b \leq sh \leq b+1$.
        In other words, $b \in \{sh-1, sh\}$.
        \item \label{ch<=b+1} (\cite{pohoata2025colorful}) $ch \leq b+1$.
        \item \label{sch<=sh+1} (\cite{pohoata2025colorful}) $sch \leq sh+1$.
        \item \label{h>=sh}  (\cite{pohoata2025colorful}) If $b = sh-1$, then $h \geq sh$, hence $h = ch = sh$.
    \end{enumerate}
\end{thm}
The proof of \cite[Theorem~1.3]{pohoata2025colorful} which actually shows items \ref{ch<=b+1} and \ref{h>=sh} above, already includes enough idea for a proof of \ref{sch<=sh+1} too.
Improving its logic, we prove that we can even bound $h \geq sch$ under the assumption $b = sh-1$.
\begin{thm}
    If $b = sh-1$, then $h \geq sch$, hence $h = ch = sh = sch$.
\end{thm}
\begin{proof}
    By Theorem~\ref{t:parameters}.\ref{h>=sh} we already know $h = sh$, thus it suffices to show $sh \geq sch$.
    Let $d = sh$ and consider $d$ families of convex sets $\F_1, \ldots, \F_d \subseteq \C$.
    If there is a transversal $T$ of $\F_1, \ldots, \F_d$ such that $\bigcap T = \varnothing$, then we are done.
    Thus let us assume $\bigcap T \neq \varnothing$ for every transversal $T$ of $\F_1, \ldots, \F_d$.




    Among all possible transversals $T$ of $\F_1, \ldots, \F_d$, take a $T$ such that $\bigcap T$ is inclusion-minimal and let $T = \{C_1, \ldots, C_d\}$.
    Note that $\size{T} = b+1$.
    By the definition of breadth, $\bigcap T = \bigcap(T \setminus \{C_i\})$ for some $i$.
    Say $C_i \in \F_i$.
    Then for every $D \in \F_i$, it must be $\bigcap(T \setminus \{C_i\}) \subseteq D$; otherwise, $(T \cup \{D\}) \setminus \{C_i\}$ is a transversal whose intersection is a proper subset of $\bigcap T$.
    Therefore, we conclude that $\bigcap T = \bigcap(T \setminus \{C_i\}) \subseteq \bigcap \F_i$.
\end{proof}

\subsection{Colorful Carath\'eodory and Tverberg}\label{subsec:strong:corollaries}

In this subsection we show some corollaries for convexity spaces where the VC dimension of $\C$ is at most $d$. Most of them are about the boundedness of colorful generalizations of the strong versions of classical invariants.

The colorful Helly theorem of Lov\'asz and B\'ar\'any \cite{barany1982generalization} states that if $\F_1, \ldots, \F_{d+1}$ are finite families of convex sets in $\Re^d$ such that for all $C_1 \in \F_1, \ldots, C_{d+1} \in \F_{d+1}$ we have $\bigcap_{i=1}^{d+1}C_i \neq \varnothing$, then there exists an $\F_j$ with $\bigcap \F_j \neq \varnothing$.
In other words, the colorful Helly number of $\cvx(\Re^d)$ is at most $d+1$.

Pohoata, Yang and Zhang prove that if the comatching number of a convexity space is $d$, then its \emph{strong} colorful Helly number is at most $d+1$ \cite{pohoata2025colorful}.
We have the following dual version. The original colorful Carath\'eodory theorem for Euclidean convexity was proved by B\'ar\'any \cite{barany1982generalization}, our version can be viewed as a strengthening of the colorful Carath\'eodory number of convexity spaces.

\begin{thm}[Strong colorful Carath\'eodory theorem]\label{t:strongcolorfulCaratheodory}
    Let $(X,\C)$ be a convexity space where the VC dimension of $\C$ is at most $d$. Let $S_1, \ldots, S_{d+1} \subseteq X$ be nonempty finite subsets.
    Then there exists a transversal $T = \{p_1, \ldots, p_{d+1}\}$ of $S_1, \ldots, S_{d+1}$ such that $\conv T \supseteq \bigcap_{i=1}^{d+1} \conv S_i$.\\
    Moreover, we can find some $i \in [d+1]$ and a colorful $d$-tuple $T' = \{p_1, \ldots, p_d\}$ of $S_1, \ldots, \hat{S_i}, \ldots, S_{d+1}$ such that $\conv T' \supseteq S_i$.
\end{thm}

\begin{proof}
    Consider all possible convex hulls $\conv \{p_1, \ldots, p_{d+1}\}$ of rainbow $(d+1)$-tuples of $S_1, \ldots, S_{d+1}$ and choose an inclusion-maximal one.
    By strong Carath\'eodory (Theorem~\ref{thm:strongEquivalences} property \ref{prop:strongCaratheodory}), we have $\conv \{p_1, \ldots, p_{d+1}\} = \conv \{p_1,$ $ \ldots, p_d\}$ without loss of generality.
    Assume that $p_{d+1} \in S_i$.
    By maximality of $\conv \{p_1, \ldots,$ $p_{d+1}\}$, we conclude that for every $q \in S_i$, $\conv \{p_1, \ldots, p_d, q\} = \conv \{p_1, \ldots, p_d\}$. This finishes the proof.
\end{proof}

With the help of the strong colorful Carath\'eodory, we can prove the following colorful Tverberg variant.

\begin{thm}[Colorful Tverberg theorem]\label{t:colorfulTverberg}
Let $(X,\C)$ be a convexity space with $\VC(\C)\le d$, and let $S_1,\ldots,S_{d+1}\subseteq X$ be pairwise disjoint finite sets. If
$|S_i|\ge d(r-1)+1$
for every $i\in[d+1]$, then there are pairwise disjoint transversals $T_1,\ldots,T_r$ such that $\bigcap_{j=1}^r\conv(T_j)\ne\varnothing$.
\end{thm}

\begin{proof}
    We repeat applying Theorem~\ref{t:strongcolorfulCaratheodory} to $(S_1, \ldots, S_{d+1})$ and update the $(d+1)$-tuple by removing the selected points from each set.

    More precisely, set $S_i^{(0)} = S_i$ for every $i$.
    Given $(S_1^{(k-1)}, \ldots, S_{d+1}^{(k-1)})$ at the $k$-th round (starting with $k = 1$), apply Theorem~\ref{t:strongcolorfulCaratheodory} to get a rainbow set $T'_k = \{p^{(k)}_1, \ldots, p^{(k)}_d\}$ and the index $f(k) \in [d+1]$ such that $S^{(k-1)}_{f(k)}$ doesn't contribute to $T'_k$.
    Define each $S_i^{(k)}$ as $S_i^{(k-1)} \setminus \{p\}$ if $p \in T'_k$ is selected from $S_i^{(k-1)}$.
    Put $S_{f(k)}^{(k)} \coloneqq S_{f(k)}^{(k-1)}$.


    We claim that the procedure can be continued until some color has been omitted $r$ times. Indeed, suppose that after $k$ rounds no color has yet been omitted $r$ times. Let $a_i(k)$ be the number of rounds among the first $k$ in which color $i$ was omitted. Then $a_i(k)\leq r-1$ for every $i$. The number of points deleted from $S_i$ is
    \[
    k-a_i(k)=\sum_{h\neq i} a_h(k)\leq d(r-1).
    \]
    Since $|S_i|\geq d(r-1)+1$, every $S_i^{(k)}$ is still nonempty. Hence the procedure can continue until, for the first time, some color has been omitted $r$ times.

    Let $1\leq j_1<\cdots<j_r$ be the first $r$ rounds in which this happens, and write $f(j_1)=\cdots=f(j_r)=\ell$.

    We have $\bigcap_{i=1}^r\conv(T'_{j_i})\supseteq S^{(j_r)}_\ell\ne\varnothing$. Since $|S_\ell|\ge d(r-1)+1\ge r$, choose distinct points $q_1,\ldots,q_r\in S_\ell$ and put $T_i=T'_{j_i}\cup\{q_i\}$. The selected $d$-tuples are disjoint by construction, and the color classes are pairwise disjoint, so the $T_i$ are pairwise disjoint transversals with a common hull point.
    \end{proof}

Pohoata et al.~\cite{pohoata2025colorful} observe that not only the bounded comatching number yields a colorful Helly theorem, but also the bounded breadth does.
Using the duality of convexity spaces, we obtain a stronger colorful Carath\'eodory theorem which uses one fewer color class under a comatching-like condition.
\begin{thm}\label{t:sharpstrongCC}
    Let $(X,\C)$ be a convexity space where the VC dimension of $\C$ is at most $d$ and assume that the convex hull of any $d$ points of $X$ in convex position is $X$ itself.
    Let $S_1, \ldots, S_d \subseteq X$ be nonempty finite subsets.
    Then there exists a transversal $T = \{p_1, \ldots, p_d\}$ of $S_1, \ldots, S_d$ such that $\conv T \supseteq \bigcap_{i=1}^d \conv S_i$.\\
    Moreover, if there is no transversal $T$ of $S_1, \ldots, S_d$ such that $\conv T = X$, then we can find some $i \in [d]$ and a rainbow ($d-1$)-tuple $T' = \{p_1, \ldots, p_{d-1}\}$ of $S_1, \ldots, \hat{S_i}, \ldots, S_d$ such that $\conv T' \supseteq S_i$.
\end{thm}

As a corollary, we get the following colorful Tverberg theorem.
\begin{thm}\label{t:sharpcolorfulTverberg}
    Let $(X,\C)$ be a convexity space where the VC dimension of $\C$ is at most $d$ and assume that the convex hull of any $d$ points of $X$ in convex position is $X$ itself.
    Let $S_1, \ldots, S_d \subseteq X$ be pairwise disjoint finite sets.
    If for every $i \in [d]$, $\size{S_i} \geq t'(d,r) \coloneqq (d-1)(r-1) + 1$, then there exist pairwise disjoint transversals $T_1, \ldots, T_r$ of $S_1, \ldots, S_d$ such that $\bigcap_{i = 1}^r \conv T_i \neq \varnothing$.
\end{thm}
The proof of Theorem~\ref{t:sharpstrongCC} and \ref{t:sharpcolorfulTverberg} is exactly the same as the proof of Theorem~\ref{t:strongcolorfulCaratheodory} and \ref{t:colorfulTverberg}, so we omit it.

\subsection{Other strengthenings of the Radon number}\label{subsec:strong:otherRadon}

If a subset $S \subseteq X$ has a partition $S_1 \cup S_2 = S$ with $\conv(S_1) \subseteq \conv(S_2)$, then it has such a partition with $S_1$ being a singleton. Thus it is tempting to define the strong Radon number as the smallest $r$ such that every subset $S \subseteq X$ of size $r$ has a Radon partition with one of the parts being a singleton. But that is a weaker assumption as the following example shows.

\begin{example}
    Let $X = \{1,2,3,4\}$ and $\C = \{\varnothing, X\} \cup (\binom{X}{3}\setminus \{\{1,2,3\}\})\cup \{ \{1,4\}, \{2,4\}, \{3,4\}, \{4\}\}$. Then there exists three points ($1,2$ and $3$) in convex position, but every three points admits a Radon partition (where one of the parts has size $1$).
\end{example}

If in addition every singleton is a convex set ($\binom{X}{1} \subseteq \C$ in notation), then a Radon partition $S_1 \cup S_2 = S$ with $|S_1| = 1$ satisfies $\conv(S_1) \subseteq \conv(S_2)$ thus existence of a Radon partition with one of the parts being a singleton implies the existence of a partition with $\conv(S_1) \subseteq \conv(S_2)$. For the other implication, observe that if we have a partition with $\conv(S_1) \subseteq \conv(S_2)$, then we have such a partition with $S_1$ being a singleton. Thus the existence of a Radon partition with one of the parts being singleton is equivalent to the other properties listed in Theorem~\ref{thm:strongEquivalences} if $\binom{X}{1}\subseteq \C$. From this point of view, it is an interesting question to analyse convexity space where one of the parts of a Radon partition has fixed size.

\begin{definition}
    For a convexity space $(X,\C)$, let $r^{(n)}_t$ be the smallest number such that any subset $Y \subseteq X$ of size $r^{(n)}_t$ can be partitioned into $t$ parts $Y = \bigcup_{i=1}^t Y_i$ such that $|Y_1|=n$ and $\bigcap_i \conv(Y_i) \neq \varnothing$.
\end{definition}

The following simple observation shows that boundedness of $r^{(n)}_t$ becomes weaker as $n$ grows.

\begin{claim}\label{cl:extendedTverbergPartition}
    Let $n,k$ be positive integers. Then $r^{(n+k)}_t \leq r^{(n)}_t+k$.
\end{claim}

\begin{proof}
    Let $S \subseteq X$ be a subset of size $r^{(n)}_t+k$, choose an arbitrary subset $T \subseteq S$ of size $r^{(n)}_t$, and let $T_1 \cup \cdots \cup T_t$ be a Tverberg-partition of $T$ with $|T_1| = n$. then $S_1 = T_1 \cup (S \setminus T)$, $S_i = T_i$ for $i \geq 2$ is a Tverberg-partition of $S$ with $|S_1| = n+k$.
\end{proof}

Unlike in the $n=1$ case, some of the restricted Radon numbers of Euclidean convexities are bounded.

\begin{claim}\label{ex:PairedRadon}
We have
\[
r^{(2)}_2\bigl(\cvx(\Re^2)\bigr)=5,
\qquad
r^{(2)}_2\bigl(\cvx(\Re^3)\bigr)=6,
\]
whereas $r^{(k)}_2(\cvx(\Re^{2k}))=\infty$ for every $k\ge1$.
\end{claim}

\begin{proof}
A simplex together with an interior point has a unique Radon partition, with the interior point as a singleton; this gives the two finite lower bounds. For the upper bounds, take a Radon partition of four points in $\Re^2$, respectively five points in $\Re^3$. If it has a singleton part, add the remaining point to that part; otherwise it already has a two-point part.

For the infinite statement, take arbitrarily many points on the moment curve in $\Re^{2k}$. Every $k$ of them span a face of the resulting cyclic polytope, so their convex hull is disjoint from the convex hull of the remaining points.
\end{proof}

The following question isolates a possible intermediate parameter between the ordinary and strong Radon numbers.

\begin{question}
Can $r_t^{(2)}$ be bounded in terms of $r_2^{(2)}$?
\end{question}

Jamison proved $r_t\le(r-1)(t-1)+1$ when $r$ is the strong Radon number~\cite{jamison1981partition}. It remains open whether the same inequality holds with $r=r_2^{(2)}$.

\section{A simplification of Bukh's example}\label{sec:Bukh}

Bukh constructed convexity spaces, for every $t\ge3$, with $r_2=4$ and $r_t>3t-2$~\cite{bukh2010radon}. We give a direct bipartite-model realization of the same obstruction on only $O(t^4)$ points, avoiding the much larger auxiliary realization in the original construction.

\medskip
\noindent\textbf{The construction.}
Let $X=A\mathbin{\dot\cup}B\mathbin{\dot\cup}C$, and let the convexity be generated by $\G=\G_2\mathbin{\dot\cup}\G_3$.
Set $A=[3t-2]$. For every pair $\{i,j\}\subseteq A$, introduce $g_{ij}\in\G_2$ with
\[
N(g_{ij})\cap A=\{i,j\}.
\]
For every triple $\{i,j,k\}\subseteq A$, introduce $g_{ijk}\in\G_3$ with
\[
N(g_{ijk})\cap A=\{i,j,k\}.
\]
For each pair $\{i,j\}\subseteq A$, introduce $c_{ij}\in C$ with
\[
N(c_{ij})=\{g_{ij}\}\cup\G_3.
\]
Finally, for every unordered pair of disjoint edges $e,f\in\binom A2$, introduce one point $b_{e,f}=b_{f,e}\in B$. If $e=\{i,j\}$ and $f=\{k,\ell\}$, define
\[
N(b_{e,f})=\{g_{ij},g_{k\ell}\}\cup
\{g_{xyz}\in\G_3:\{x,y,z\}\cap(e\cup f)\ne\varnothing\}.
\]
Thus
\[
|B|=3\binom{3t-2}{4},
\]
and the ground set has size $O(t^4)$.

\begin{figure}
\centering

\newcommand{\scaleratio}{0.35}
\begin{minipage}{0.3\textwidth}
\centering
\scalebox{\scaleratio}{
\begin{tikzpicture}[every node/.style={font=\small},>=stealth]
\draw (0,4) ellipse (5 and 1.2);
\node at (-3,5.4) {\textbf{A}};
\node at (1.2,5.4) {\textbf{B}};
\node at (3.5,5.4) {\textbf{C}};
\draw (0,5.2) -- (0,2.8);
\draw (2.5,5) -- (2.5,3);
\node[circle,draw,inner sep=2pt] (i) at (-4.5,4) {$i$};
\node[circle,draw,inner sep=2pt] (j) at (-3.5,4) {$j$};
\node[circle,draw,inner sep=2pt] (k) at (-2.5,4) {$k$};
\node[circle,draw,inner sep=2pt] (l) at (-1.5,4) {$\ell$};
\draw (0,0) ellipse (5 and 1.2);
\node at (-2.5,-1.3) {\textbf{$\G_2$}};
\node at (2.5,-1.3) {\textbf{$\G_3$}};
\draw (0,1.2) -- (0,-1.2);
\node[circle,draw,inner sep=2pt] (gij) at (-2,0) {$g_{ij}$};
\node[circle,draw,inner sep=2pt] (gjkl) at (2,0) {$g_{jk\ell}$};
\draw[-] (i) -- (gij);
\draw[-] (j) -- (gij);
\draw[-] (j) -- (gjkl);
\draw[-] (k) -- (gjkl);
\draw[-] (l) -- (gjkl);
\end{tikzpicture}}
\end{minipage}%
\hfill
\begin{minipage}{0.3\textwidth}
\centering
\scalebox{\scaleratio}{
\begin{tikzpicture}[every node/.style={font=\small},>=stealth]
\draw (0,4) ellipse (5 and 1.2);
\node at (-3.5,5.4) {\textbf{A}};
\node at (0,5.4) {\textbf{B}};
\node at (3.5,5.4) {\textbf{C}};
\draw (-2.5,5) -- (-2.5,3);
\draw (2.5,5) -- (2.5,3);
\draw (0,0) ellipse (5 and 1.2);
\node at (-2.5,-1.3) {\textbf{$\G_2$}};
\node at (2.5,-1.3) {\textbf{$\G_3$}};
\draw (0,1.2) -- (0,-1.2);
\node[circle,draw,inner sep=2pt] (b) at (-0,4) {$b_{ij,k\ell}$};
\node[circle,draw,inner sep=2pt] (gij) at (-3,0) {$g_{ij}$};
\node[circle,draw,inner sep=2pt] (gkl) at (-1,0) {$g_{k\ell}$};
\node[circle,draw,inner sep=2pt] (gijk) at (1,0) {$g_{ijk}$};
\node[circle,draw,inner sep=2pt] (gjkx) at (2,0) {$g_{jkx}$};
\node[circle,draw,inner sep=2pt] (giyz) at (3,0) {$g_{iyz}$};
\node[circle,draw,inner sep=2pt] (gxyz) at (4,0) {$g_{xyz}$};
\draw[-] (b) -- (gij);
\draw[-] (b) -- (gkl);
\draw[-] (b) -- (gijk);
\draw[-] (b) -- (gjkx);
\draw[-] (b) -- (giyz);
\end{tikzpicture}}
\end{minipage}%
\hfill
\begin{minipage}{0.3\textwidth}
\centering
\scalebox{\scaleratio}{
\begin{tikzpicture}[every node/.style={font=\small},>=stealth]
\draw (0,4) ellipse (5 and 1.2);
\node at (-3.5,5.4) {\textbf{A}};
\node at (-1.2,5.4) {\textbf{B}};
\node at (3.5,5.4) {\textbf{C}};
\draw (-2.5,5) -- (-2.5,3);
\draw (0,5.2) -- (0,2.8);
\draw (0,0) ellipse (5 and 1.2);
\node at (-2.5,-1.3) {\textbf{$\G_2$}};
\node at (2.5,-1.3) {\textbf{$\G_3$}};
\draw (0,1.2) -- (0,-1.2);
\node[circle,draw,inner sep=2pt] (c) at (2,4) {$c_{ij}$};
\node[circle,draw,inner sep=2pt] (gij) at (-3,0) {$g_{ij}$};
\node[circle,draw,inner sep=6pt] (g1) at (1,0) {};
\node[circle,draw,inner sep=6pt] (g2) at (2,0) {};
\node[circle,draw,inner sep=6pt] (g3) at (3,0) {};
\draw[-] (c) -- (g1);
\draw[-] (c) -- (g2);
\draw[-] (c) -- (g3);
\draw[-] (c) -- (gij);
\end{tikzpicture}}
\end{minipage}

\caption{An illustration of the bipartite graph model of Bukh's construction.}

\end{figure}
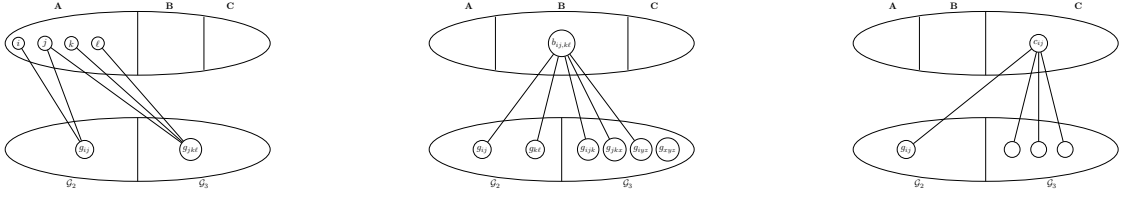

For distinct $i,j,k\in A$,
\[
\conv(\{i\})=\{i\},\qquad
\conv(\{i,j\})=g_{ij},\qquad
\conv(\{i,j,k\})=g_{ijk}.
\]
Moreover, if $e,f\in\binom A2$ are disjoint, then
\[
\conv(e)\cap\conv(f)=\{b_{e,f}\}.
\]
These identities are the only features of the construction needed for the Tverberg obstruction.

\medskip
\noindent\textbf{Lower bound on the Tverberg number.}
We claim that $A$ has no Tverberg $t$-partition. If one part were a singleton $\{a\}$, then $a$ would have to lie in the hull of every other part. Among the remaining $t-1$ parts, one has size at most three, but no point of $A$ lies in the hull of at most three other points of $A$. Hence every part has size at least two.

No point lies in the hulls of three pair-parts: an $A$-point is incident with at most two of three disjoint pair generators, a point of $C$ is incident with only one generator from $\G_2$, and a point of $B$ with only two. Thus at most two parts have size two. Since $|A|=3t-2$, there must be exactly two pair-parts and every other part must be a triple. If the two pairs are $e$ and $f$, their hulls meet only at $b_{e,f}$. Every triple part is disjoint from $e\cup f$, so its generator is not adjacent to $b_{e,f}$; hence $b_{e,f}$ does not lie in its hull. Therefore the partition has no common Tverberg point.

\subsection{Upper bound on the Radon number.}
We now check $r_2 \leq 4$. An important claim covering many cases is the following.

\begin{claim}
If $x,y,u,v\in X$ are four distinct points and $N(\{x,y\})\subseteq\G_3$, then
\[
\conv(\{x,y\})\cap\conv(\{u,v\})\ne\varnothing.
\]
\end{claim}

\begin{proof}
    We have $|N(\{u,v\}) \cap \G_2|\leq 1$ for any two $u,v \in X$.
    
    If $g_{ij} \in N(\{u,v\}) \cap \G_2$, then $c_{ij} \in \conv(\{x,y\}) \cap \conv(\{u,v\}) = N(N(\{x,y\})\cup N(\{u,v\}))$, otherwise $N(\{u,v\}) \subseteq \G_3$ and even $C \subseteq \conv(\{x,y\}) \cap \conv(\{u,v\})$ is true.
\end{proof}

Let $Q \in \binom{X}{4}$. We want to show that it has a Radon partition. 

\textbf{Assume from now on that every pair from $Q$ has a common neighbor in $\G_2$.}
Otherwise there is a pair $\{x,y\} \subseteq Q$ with $N(\{x,y\}) \subseteq \G_3$, and then the claim above yields a Radon partition of $Q$.

This assumption implies in particular that $|Q\cap C| \le 1$. If $c_{ij}\in Q\cap C$, then every other point of $Q$ is adjacent to $g_{ij}$, hence
\[
N(Q\setminus\{c_{ij}\})\subseteq \{g_{ij}\}\cup \G_3 \subseteq N(c_{ij}),
\]
so $c_{ij}\in \conv(Q\setminus\{c_{ij}\})$, and $Q$ is Radon.

\textbf{Assume from now on that $Q\cap C=\varnothing$.}

If $Q = \{i,j,k,\ell\} \subseteq A$ with $i < j < k < \ell$, then
\[
b_{ij,k\ell } \in \conv(\{i,j\}) \cap \conv(\{k,\ell\}),
\]
so $Q$ is Radon.

\textbf{Assume from now on that $Q\cap B\neq\varnothing$.}
For $x\in X$, write
\[
E(x):=N(x)\cap \G_2.
\]
Thus $E(a)=\{g_{ab}:b\in A\setminus\{a\}\}$ for $a\in A$, while
$E(b_{ij,k\ell})=\{g_{ij},g_{k\ell}\}$.

If $Q$ contains exactly one point of $B$, say $b_{ij,k\ell}\in Q\cap B$. Then every point of $Q\setminus\{b_{ij,k\ell}\}\subseteq A$
must share a $\G_2$-neighbor with $b_{ij,k\ell}$, so
\[
Q\setminus\{b_{ij,k\ell}\}\subseteq \{i,j,k,\ell\}.
\]
Hence one of the pairs $\{i,j\}$ or $\{k,\ell\}$ is contained in $Q$, and since
$b_{ij,k\ell}$ lies in the convex hull of that pair, we obtain a Radon partition
\[
Q=\{i,j\}\sqcup (Q\setminus\{i,j\}).
\]

\textbf{Assume from now on that $Q$ contains at least two points of $B$.}
Then the sets $E(\beta)$, $\beta\in Q\cap B$, are pairwise intersecting 2-subsets of $\G_2$, hence they
have a common element, say $g_{uv}$. 

Indeed, for two sets this is trivial; for
three the only other possibility is the triangle pattern
$\{e_1,e_2\},\{e_1,e_3\},\{e_2,e_3\}$, which is impossible here because if
$Q$ also contains a point $a\in A$, then the star $E(a)$ cannot meet all three
generators $e_1,e_2,e_3 \in \G_2$; and for four 2-sets, pairwise
intersection already forces a common element.

If $Q\subseteq B$, then $c_{uv}$ is adjacent to $g_{uv}$ and to every element of $\G_3$, so
$c_{uv}$ lies in the convex hull of every pair of points of $Q$.
Hence any $2+2$ partition of $Q$ is a Radon partition.

\textbf{Assume from now on that $Q$ contains a point of $A$.}

If $Q$ contains some $a\in A\cap \{u,v\}$, then choose distinct $\beta_1,\beta_2\in Q\cap B$.
If also $\{u,v\}\subseteq Q$, then
\[
c_{uv}\in \conv(\{u,v\})\cap \conv(\{\beta_1,\beta_2\}),
\]
so $Q$ is Radon.
Otherwise every common neighbor of $a$ and $\beta_1$ is adjacent to $\beta_2$:
the only one in $\G_2$ is $g_{uv}$, and every one in $\G_3$ contains $a$.
Therefore
\[
\beta_2\in \conv(\{a,\beta_1\}),
\]
and $Q$ is again Radon.

\textbf{Assume finally that $Q\cap A\cap \{u,v\}=\varnothing$.}
Then every $a\in Q\cap A$ shares a $\G_2$-neighbor with every point of $Q\cap B$,
but not via $g_{uv}$.
Hence for each $\beta\in Q\cap B$, the second element of $E(\beta)$ must be incident with $a$.
If there were two distinct points $a,b\in Q\cap A$, this would force that second
element to be $g_{ab}$ for every $\beta\in Q\cap B$, which is impossible since $|Q\cap B|\ge 2$
and distinct points of $B$ have distinct $\G_2$-neighborhoods.

So $Q=\{a,\beta_1,\beta_2,\beta_3\}$ with $a\in A$, and each
\[
E(\beta_i)=\{g_{uv},g_{ax_i}\}.
\]
In particular, $a,\beta_2,\beta_3$ have no common neighbor in $\G_2$, so every common
neighbor of these three points lies in $\G_3$ and contains $a$.
Since $\beta_1$ is adjacent to every $g_{ijk} \in \G_3 \cap N(a)$, we get
\[
\beta_1\in \conv(\{a,\beta_2,\beta_3\}),
\]
and hence $Q$ is Radon.

\section{Tverberg numbers of $S_3$-separable convexity spaces}
\label{sec:separable}

We first isolate the two ingredients used throughout the section: a
Helly-type centerpoint whose containing halfspaces are all deep, and
the conversion of a net for those halfspaces into a convex-hull
certificate.  When convenient, finite point sets below may be regarded
as labelled multisets; cardinalities are then counted with
multiplicity, while repeated copies do not affect the convex hull.

\begin{lem}
\label{lem:helly-centerpoint}
Let $(X,\C)$ be an $S_3$-separable convexity space with Helly number
$h$, let $\B$ be its family of halfspaces, and let $P$ be a nonempty
finite point multiset.  There is a point $x\in X$ such that:
\begin{enumerate}
    \item\label{item:centerpoint-depth}
    every $B\in\B$ containing $x$ satisfies
    $|B\cap P|\ge \frac{|P|}{h};$
    \item\label{item:net-to-hull}
    if $A\subseteq P$ meets every halfspace in $\B$ containing $x$,
    then $x\in\conv(A)$.
\end{enumerate}
In particular, every $\varepsilon$-net for the traces of $\B$ on $P$
with $\varepsilon\le 1/h$ has convex hull containing $x$.
\end{lem}

\begin{proof}
Write $n=|P|$ and consider the finite family
\[
\mathcal H_P=
\left\{
\conv(A):
A\subseteq P,\quad
|A|>\left(1-\frac1h\right)n
\right\}.
\]
Any $h$ submultisets occurring in this definition have a common
labelled point, since the union of their complements has size less
than $n$.  Hence every $h$ members of $\mathcal H_P$ intersect, and
the Helly property gives
\[
x\in\bigcap\mathcal H_P.
\]

Suppose that $B\in\B$ contains $x$ but $|B\cap P|<n/h$.  Then
$A=P\setminus B$ has size greater than $(1-1/h)n$.  Since
$X\setminus B$ is convex,
\[
x\in\conv(A)\subseteq X\setminus B,
\]
a contradiction.  This proves~\ref{item:centerpoint-depth}.

For~\ref{item:net-to-hull}, suppose that $x\notin\conv(A)$.
By $S_3$-separability, there is a halfspace $H$ such that
$\conv(A)\subseteq H$
and 
$x\notin H.$
The complementary halfspace $X\setminus H$ contains $x$ and is
disjoint from $A$, contrary to the hypothesis on $A$.
\end{proof}

We shall use the standard probabilistic $\varepsilon$-net theorem in
the following packed form.

\begin{lem}
\label{lem:disjoint-epsilon-nets}
There is an absolute constant $c_0$ with the following property.
Let $(V,\mathcal R)$ be a finite range space of VC-dimension at most
$d$, let $0<\varepsilon\le 1/2$, and put
$m_0=
\left\lceil
c_0\frac d\varepsilon\log\frac{2}{\varepsilon}
\right\rceil.$
Then:
\begin{enumerate}
    \item\label{item:random-net}
    for every $m$ with $m_0\le m\le |V|$, a uniformly random
    $m$-element subset of $V$ is an $\varepsilon$-net with probability
    at least $3/4$;
    \item\label{item:disjoint-nets}
    if $|V|\ge 2m_0q$, then $V$ contains $q$ pairwise disjoint
    $\varepsilon$-nets.
\end{enumerate}
\end{lem}

\begin{proof}
Part~\ref{item:random-net} is the standard probabilistic
$\varepsilon$-net theorem, after increasing the absolute constant.
For~\ref{item:disjoint-nets}, take a uniformly random permutation of
$V$ and divide its first $2m_0q$ elements into consecutive
$m_0$-element blocks
$Z_1,\ldots,Z_{2q}.$
Each $Z_i$ is marginally a uniformly random $m_0$-element subset of
$V$.  By~\ref{item:random-net}, the expected number of blocks that are
$\varepsilon$-nets is at least
$\frac34\cdot 2q>q.$
Consequently, some choice of the blocks contains at least $q$
pairwise disjoint $\varepsilon$-nets.
\end{proof}

\begin{thm}\label{thm:Helly-VC-Tverberg}
Let $(X,\C)$ be an $S_3$-separable convexity space with Helly number
$h$, and let $\B$ be its family of halfspaces.  If $\VC(\B)\le d$,
then
$r_t=O\bigl(dh\log h\bigr)t.$
\end{thm}

\begin{proof}
Set $\varepsilon=1/(2h)$ and let $m_0$ be as in
Lemma~\ref{lem:disjoint-epsilon-nets}.  Take an arbitrary finite
$Y\subseteq X$ with
$|Y|\ge 2m_0t.$
Choose $x$ as in Lemma~\ref{lem:helly-centerpoint}, and apply
Lemma~\ref{lem:disjoint-epsilon-nets} to the traces of $\B$ on $Y$.
We obtain pairwise disjoint $\varepsilon$-nets
$Y_1,\ldots,Y_t\subseteq Y.$
Every halfspace containing $x$ contains at least $|Y|/h$ points of
$Y$, and hence is met by every $Y_i$. 
Lemma~\ref{lem:helly-centerpoint} therefore gives
$x\in\bigcap_{i=1}^t\conv(Y_i).$
Distribute the points of $Y\setminus\bigcup_iY_i$ arbitrarily among
the $Y_i$.  The resulting parts form a Tverberg partition.  Since
$m_0=O\bigl(dh\log h\bigr),$
the theorem follows.
\end{proof}

\begin{cor}[Restatement of Theorem~\ref{thm:separable}]
If an $S_3$-separable convexity space has Radon number $r$, then
$r_t=O(r^2\log r)t.$
\end{cor}

\begin{proof}
Levi's inequality gives $h\le r-1$.  Moreover,
$\VC(\B)\le r-1$: if an $r$-element set were shattered by halfspaces,
then every nontrivial partition of it would be separated by a pair of
complementary halfspaces, and hence could not be a Radon partition.
The result follows from
Theorem~\ref{thm:Helly-VC-Tverberg}.
\end{proof}

\section{Discussion}\label{sec:discussion}
The central quantitative problem is to determine the optimal dependence of $r_t$ on the ordinary Radon number. Even under separation assumptions, it is open whether a universal estimate of order $O(rt)$ holds. Our $S_3$ bound leaves a polynomial gap in $r$. A second, potentially more flexible route is the restricted parameter $r_2^{(2)}$: it is not known whether $r_2^{(2)}<\infty$ alone forces any bound on $r_t^{(2)}$.

Bounded strong Radon number, equivalently bounded VC-dimension of the convex sets, supports exact certificates and colorful theorems that have no analogue for the ordinary Radon number. Two decisive tests of the strength of this hypothesis are to determine sharp piercing bounds under the $(p,q)$-condition and to decide whether bounded VC-dimension in a convexity space forces a bounded Leray number; see~\cite[Question~4.12]{pohoata2025colorful}.

Finally, the incidence model suggests developing duality beyond finite spaces. The basis-dual Radon number is not controlled by the ordinary Radon number in general, but the obstruction in Claim~\ref{cl:noDualRadonBound} is decomposable. It would be valuable to identify natural irreducibility or connectivity assumptions under which the two parameters become quantitatively related, and to determine which generator-relative dual notions admit an intrinsic formulation.

\section*{Acknowledgement}

The third author would like to thank Istv\'an Tomon for bringing the strong Helly number to his attention.

\bibliographystyle{plain}

\bibliography{biblio}

\appendix
\section{Appendix: Selection, weak nets, and piercing consequences}
\label{subsec:separable:consequences}

The centerpoint--net principle, Lemma~\ref{lem:helly-centerpoint}, also yields selection and piercing
statements.  Throughout this subsection, $(X,\C)$ is an
$S_3$-separable convexity space, $h$ is its Helly number, $\B$ is its
family of halfspaces, and
$d=\VC(\B).$
Fix a sufficiently large absolute constant $C_0$ and put
\begin{equation}\label{eq:selection-parameter}
a=\left\lceil C_0dh\log h\right\rceil.
\end{equation}

We first record how Theorem~\ref{thm:Helly-VC-Tverberg} complements
the colorful $k$-wise method of Keller and
Smorodinsky~\cite{keller2026colorful}.

\begin{prop}[Interpolation with the colorful $k$-wise method]
\label{prop:KS-interpolation}
Suppose in addition that $(X,\C)$ is $S_4$-separable, and let $r$ be
its Radon number.  Then
\[
r_t=O\left(
rt\min\left\{
h\log(2h),\,
\min\{h,t\}\log(2rt)
\right\}
\right).
\]
\end{prop}

\begin{proof}
Since $d\le r-1$, Theorem~\ref{thm:Helly-VC-Tverberg} gives
$r_t=O\bigl(rh\log(2h)\bigr)t.$
On the other hand, the uncolored $k$-wise Tverberg theorem of Keller
and Smorodinsky gives, from
$O\bigl(krt\log(2rt)\bigr)$
points, a partition into $t$ parts such that every $k$ of their convex
hulls intersect.  Take $k=\min\{h,t\}$.  If $h\le t$, the Helly
property upgrades $h$-wise intersection to total intersection.  If
$t<h$, then $k=t$, so the conclusion already says that all $t$ hulls
meet.  Hence
$r_t=O\bigl(r\min\{h,t\}\,t\log(2rt)\bigr).$
Taking the better of the two estimates proves the proposition.
\end{proof}

For colorful sampling we need the following product version of the
probabilistic $\varepsilon$-net theorem.

\begin{lem}[Rainbow $\varepsilon$-net lemma]
\label{lem:rainbow-epsilon-net}
There is an absolute constant $C_1$ with the following property.  Let
$\mathcal R$ be a range space of VC-dimension at most $d$, and let
$P_1,\ldots,P_m$ be pairwise disjoint nonempty finite subsets of its
ground set.  For $R\in\mathcal R$, put
\[
\mu(R)=\frac1m\sum_{i=1}^m
\frac{|R\cap P_i|}{|P_i|}.
\]
If $0<\varepsilon\le1$ and
$m\ge C_1\frac d\varepsilon\log\frac2\varepsilon,$
then a uniformly random transversal
$T=\{p_1,\ldots,p_m\}$, with the $p_i\in P_i$ chosen independently,
meets every $R\in\mathcal R$ satisfying $\mu(R)\ge\varepsilon$ with
probability at least $1/2$.
\end{lem}

\begin{proof}
Let $T'=\{p'_1,\ldots,p'_m\}$ be an independent transversal,
and let $\mathcal E$ be the event that $T$ misses some
$R\in\mathcal R$ with $\mu(R)\ge\varepsilon$.  Conditional on $T$ and
on a chosen witness $R$, the random variable $|R\cap T'|$ is a sum of
independent Bernoulli variables with expectation at least
$\varepsilon m$.  By the Chernoff bound,
$\Pr\left(
|R\cap T'|<\frac{\varepsilon m}{2}
\right)
\le
\exp\left(-\frac{\varepsilon m}{8}\right)
\le\frac12,$
after increasing $C_1$.  Therefore
\begin{equation}\label{eq:symmetrization-rainbow}
\Pr(\mathcal E)\le 2\Pr(\mathcal E'),
\end{equation}
where $\mathcal E'$ is the event that some $R\in\mathcal R$ satisfies
\[
R\cap T=\varnothing
\qquad\text{and}\qquad
|R\cap T'|\ge\frac{\varepsilon m}{2}.
\]

Condition on the unordered labelled pairs
$\{p_i,p'_i\}$ and then expose which member of each pair belongs to
$T$.  On the resulting multiset of at most $2m$ labelled points, the
number of traces of members of $\mathcal R$ is at most
$\sum_{j=0}^d\binom{2m}{j}
\le
\left(\frac{2em}{d}\right)^d$
by the Sauer--Shelah lemma.  A fixed trace witnessing $\mathcal E'$
contains exactly one member of at least $\varepsilon m/2$ of the
pairs, and all these members must be assigned to $T'$.  The
probability of this event is at most $2^{-\varepsilon m/2}$.  Hence
\[
\Pr(\mathcal E')
\le
\left(\frac{2em}{d}\right)^d2^{-\varepsilon m/2}
\le\frac14
\]
for
$m\ge C_1(d/\varepsilon)\log(2/\varepsilon)$ and $C_1$ sufficiently
large.  Together with~\eqref{eq:symmetrization-rainbow}, this proves
the lemma.
\end{proof}

We can now strengthen the selection lemmas of Keller and
Smorodinsky~\cite{keller2026colorful}.

\begin{thm}[Selection lemma]\label{thm:separable-selection}
Let $P\subseteq X$ be finite with $n=|P|\ge a$.  There is a point
$x\in X$ such that
\[
\left|
\left\{
A\in\binom{P}{a}:x\in\conv(A)
\right\}
\right|
\ge \frac12\binom{n}{a}.
\]
The same statement holds for a labelled point multiset, with
cardinalities counted with multiplicity.
\end{thm}

\begin{proof}
Choose $x$ from Lemma~\ref{lem:helly-centerpoint}.  By
Lemma~\ref{lem:disjoint-epsilon-nets}.\ref{item:random-net}, with
$\varepsilon=1/(2h)$ and the choice of $a$ in
\eqref{eq:selection-parameter}, a uniformly random
$A\in\binom{P}{a}$ meets every halfspace containing $x$ with
probability at least $1/2$.  Lemma~\ref{lem:helly-centerpoint} then gives $x\in\conv(A)$.
\end{proof}

\begin{thm}[Colorful selection lemma]
\label{thm:separable-colorful-selection}
Let $P_1,\ldots,P_a\subseteq X$ be pairwise disjoint finite sets of a
common size $n\ge1$.  There is a point $x\in X$ such that
\[
\left|
\left\{
(p_1,\ldots,p_a)\in P_1\times\cdots\times P_a:
x\in\conv\{p_1,\ldots,p_a\}
\right\}
\right|
\ge \frac12 n^a.
\]
\end{thm}

\begin{proof}
Apply Lemma~\ref{lem:helly-centerpoint} to
$P=P_1\sqcup\cdots\sqcup P_a$
and let $x$ be the resulting point.  If $B\in\B$ contains $x$, then
\[
\frac1a\sum_{i=1}^a\frac{|B\cap P_i|}{n}
=
\frac{|B\cap P|}{an}
\ge\frac1h.
\]
Apply Lemma~\ref{lem:rainbow-epsilon-net} with
$\varepsilon=1/(2h)$ to the halfspaces containing $x$.  By the choice
of $a$, a uniformly random rainbow transversal meets every such
halfspace with probability at least $1/2$. 
Lemma~\ref{lem:helly-centerpoint} then implies that its convex hull
contains $x$.
\end{proof}

The colorful selection lemma gives a Tverberg tradeoff in which the
number of colors is independent of the number of desired parts.

\begin{cor}[A second colorful Tverberg tradeoff]
\label{cor:separable-colorful-tverberg-tradeoff}
Let $t\ge1$, and let $P_1,\ldots,P_a\subseteq X$ be pairwise disjoint
finite sets satisfying
\[
|P_i|\ge 2a(t-1)+1
\qquad\text{for every }i\in[a].
\]
Then there are pairwise disjoint rainbow sets
$A_1,\ldots,A_t$, each containing exactly one point from every color
class, such that
\[
\bigcap_{j=1}^t\conv(A_j)\ne\varnothing.
\]
\end{cor}

\begin{proof}
Set $M=2a(t-1)+1$, and replace each $P_i$ by an arbitrary
$M$-element subset.  By
Theorem~\ref{thm:separable-colorful-selection}, there is a point
$x\in X$ contained in the hulls of at least $M^a/2$ rainbow
transversals.  Form the $a$-partite $a$-uniform hypergraph whose edges
are precisely these transversals.

If this hypergraph had no matching of size $t$, a maximal matching
would have at most $t-1$ edges, and the union of those edges would be
a vertex cover of size at most $a(t-1)$.  Since each vertex belongs to
at most $M^{a-1}$ rainbow transversals, the hypergraph would have at
most
\[
a(t-1)M^{a-1}<\frac12M^a
\]
edges, a contradiction.  Hence it has a matching
$A_1,\ldots,A_t$ of size $t$, and every $\conv(A_j)$ contains $x$.
\end{proof}

We next record the weak-net consequence.  A set $N\subseteq X$ is a
\emph{weak $\varepsilon$-net} for a finite point set $P\subseteq X$ if
\[
N\cap\conv(Q)\ne\varnothing
\]
for every $Q\subseteq P$ with $|Q|\ge\varepsilon|P|$.  The same
definition applies to labelled multisets, with sizes counted with
multiplicity.

\begin{cor}[Weak $\varepsilon$-nets]
\label{cor:separable-weak-epsilon-net}
For every finite point set or labelled point multiset $P$ and every
$0<\varepsilon\le1$, there is a weak $\varepsilon$-net
$N\subseteq X$ satisfying
\[
|N|\le 2\left(\frac e\varepsilon\right)^a.
\]
\end{cor}

\begin{proof}
Write $n=|P|$ and $s=\lceil\varepsilon n\rceil$.  If $s<a$, take
$N$ to be the support of $P$.  Then
$|N|\le n<\frac a\varepsilon
\le 2\left(\frac e\varepsilon\right)^a.$
Assume that $s\ge a$.  We construct $N$ greedily.  Call an
$a$-submultiset $A$ of $P$ alive if
$N\cap\conv(A)=\varnothing.$
If the current $N$ is not a weak $\varepsilon$-net, there is a
submultiset $Q\subseteq P$ with $|Q|\ge s$ and
$N\cap\conv(Q)=\varnothing.$
By Theorem~\ref{thm:separable-selection}, some point $x\in X$ belongs
to the hulls of at least
$\frac12\binom{|Q|}{a}
\ge\frac12\binom{s}{a}$
$a$-submultisets of $Q$.  All of them are alive before $x$ is added to
$N$, and none is alive afterwards.  Thus each greedy step kills at
least $\frac12\binom{s}{a}$ alive submultisets.  Since initially
there are $\binom{n}{a}$ of them, the process stops after at most
$2\frac{\binom{n}{a}}{\binom{s}{a}}
\le
2\left(\frac{en}{s}\right)^a
\le
2\left(\frac e\varepsilon\right)^a$
steps.  At termination, $N$ is a weak $\varepsilon$-net.
\end{proof}

Finally, we insert the improved weak-net exponent into the standard
Alon--Kleitman scheme.  We state the result in terms of arbitrary
fractional Helly data.  Say that $q_0$ is a fractional Helly number
with function $\beta:(0,1]\to(0,1]$ if, whenever at least an
$\alpha$-fraction of the $q_0$-subfamilies of a finite family of
convex sets intersect, some point belongs to at least a
$\beta(\alpha)$-fraction of the family.

\begin{cor}[A quantitative $(p,q)$-theorem]
\label{cor:separable-pq}
Assume that $\C$ has fractional Helly number at most $q_0$ with
fractional Helly function $\beta$.  Let $p\ge q\ge q_0$, and let
$\F\subseteq\C$ be a finite family of nonempty convex sets satisfying
the $(p,q)$-property.  Put
\[
p'=(p-1)(q-1)+1,
\qquad
\alpha=\frac{\binom q{q_0}}{\binom{p'}{q_0}},
\qquad
\gamma=\beta(\alpha).
\]
Then $\F$ has a transversal of size at most
$2\left(\frac e\gamma\right)^a.$
\end{cor}

\begin{proof}
We first prove a weighted consequence of fractional Helly.  Let
$\F'$ be an arbitrary finite labelled multiset whose members are
taken from $\F$.  Repeating all labels equally if necessary, assume
that $n=|\F'|\ge p'$.

Every $p'$ labelled members of $\F'$ contain $q$ members with
nonempty intersection.  Indeed, either they contain at least $p$
distinct members of $\F$, in which case the $(p,q)$-property applies,
or some member occurs at least
$\left\lceil\frac{p'}{p-1}\right\rceil=q$
times.  Let $M_{q_0}$ be the number of intersecting labelled
$q_0$-subfamilies of $\F'$.  Counting pairs $(I,J)$, where $J$ is a
$p'$-subfamily and $I\subseteq J$ is an intersecting
$q_0$-subfamily, gives
$M_{q_0}\binom{n-q_0}{p'-q_0}
\ge
\binom n{p'}\binom q{q_0}.$
Since
$\binom n{p'}\binom{p'}{q_0}
=
\binom n{q_0}\binom{n-q_0}{p'-q_0},$
we obtain
$\frac{M_{q_0}}{\binom n{q_0}}
\ge
\frac{\binom q{q_0}}{\binom{p'}{q_0}}
=
\alpha.$
The fractional Helly property therefore yields a point contained in
at least $\gamma n$ members of $\F'$, counted with multiplicity.

It follows, by approximation, that for every probability distribution
on the members of $\F$, some point is contained in sets of total
weight at least $\gamma$.  There are only finitely many incidence
patterns of points on the finite family $\F$.  Applying the minimax
theorem to the resulting finite zero--one matrix gives a finitely
supported probability distribution $\mu$ on $X$ such that
\[
\mu(F)\ge\gamma
\qquad\text{for every }F\in\F.
\]
Equivalently, this is the usual finite-dimensional linear-programming
duality argument.  Since the corresponding finite game has a rational
optimal strategy, we may clear denominators and obtain a finite
labelled multiset $Y$ of points satisfying
\begin{equation}\label{eq:fractional-piercing-measure}
|F\cap Y|\ge\gamma|Y|
\qquad\text{for every }F\in\F.
\end{equation}

Apply Corollary~\ref{cor:separable-weak-epsilon-net} to $Y$ with
$\varepsilon=\gamma$.  We obtain a set $N\subseteq X$ of size at most
$2(e/\gamma)^a$ which meets the convex hull of every submultiset of
$Y$ of size at least $\gamma|Y|$.  For every $F\in\F$,
\eqref{eq:fractional-piercing-measure} and the convexity of $F$ give
$\conv(F\cap Y)\subseteq F.$
Hence $N\cap F\ne\varnothing$ for every $F\in\F$, so $N$ is the
required transversal.
\end{proof}

By the fractional Helly theorem for $S_3$-separable convexity
spaces~\cite{holmsen2024helly}, a space of Radon number $r$ admits
fractional Helly data with $q_0\le 2^r$.  Thus
Corollary~\ref{cor:separable-pq} applies to every such space.  More
importantly for the quantitative comparison, all the consequences
above use the parameter
\[
a=O\bigl(dh\log h\bigr)
=
O(r^2\log r).
\]
For $S_4$-separable spaces this improves the $O(r^3)$ parameter in
the selection, colorful selection, weak-net, $(p,q)$, and second
colorful Tverberg results of Keller and
Smorodinsky~\cite{keller2026colorful}.  Their adaptation to
$S_3$-separation gives the corresponding parameter
$O(r^4\log r)$ under an additional compactness assumption; the
arguments above require neither that assumption nor separation of two
disjoint convex sets.

\end{document}